\documentclass[10pt]{amsart}
\theoremstyle{plain}
\usepackage{amsmath, amsfonts, amsthm, amssymb,amscd,bbold, hyperref}
\usepackage{graphicx}
\usepackage{color}
\usepackage{todonotes}
\usepackage[all]{xy}
\graphicspath{ {Figures/} }

\title[]{On braided, banded surfaces and ribbon obstructions}

\author{J. Elisenda Grigsby}
\thanks{This work was partially supported by NSF CAREER award DMS-1151671 and a grant from the Simons Foundation (\#396324 to J. Elisenda Grigsby).}
\address{Boston College; Department of Mathematics; 522 Maloney Hall; Chestnut Hill, MA 02467}
\email{grigsbyj@bc.edu}

\theoremstyle{}
\newtheorem{theorem}{Theorem}
\newtheorem{lemma}[theorem]{Lemma}
\newtheorem{proposition}[theorem]{Proposition}
\newtheorem{corollary}[theorem]{Corollary}
\newtheorem{conjecture}[theorem]{Conjecture}

\newtheorem{definition}[theorem]{Definition}
\newtheorem{remark}[theorem]{Remark}

\newcommand{\R}{\ensuremath{\mathbb{R}}}
\newcommand{\Z}{\ensuremath{\mathbb{Z}}}

\newcommand{\Id}{\ensuremath{\mbox{\textbb{1}}}}
\newcommand{\rk}{\ensuremath{\mbox{rk}_n}}

\newcommand{\Braid}{\ensuremath{\mathfrak B}}

\newcommand{\hbeta}{\widehat{\beta}_{S^3}}
\newcommand{\hbetaB}{\widehat{\beta}_{S^3 \setminus B}}

\begin{document}
\bibliographystyle{plain}
\begin{abstract} We discuss how to apply work of L. Rudolph to braid conjugacy class invariants to obtain potentially effective obstructions to a slice knot being ribbon. We then apply these ideas to a family of braid conjugacy class invariants \cite{dt} coming from Khovanov-Lee theory and explain why we do not obtain effective ribbon obstructions in this case.
\end{abstract}

\maketitle
\bibliographystyle{plain}

\section{Introduction} \label{sec:intro}
Recall that a knot $K \subseteq (S^3 = \partial B^4)$ is said to be {\em slice}\footnote{We work in the smooth category throughout.} if $K$ bounds a properly imbedded disk in $B^4$.  $K$ is said to be {\em ribbon} if $K$ bounds an immersed disk in $S^3$ whose singularities are all of ribbon type, as in Figure \ref{fig:ribbon}.

By pushing neighborhoods of the ribbon singularities into $B^4$, one sees that every ribbon knot is slice. The converse statement is a long-standing conjecture of Fox \cite{Fox_Probs}:

\begin{conjecture}[Slice-Ribbon conjecture] If $K$ is slice, then $K$ is ribbon.
\end{conjecture}

There are a number of proposed counterexamples, cf. \cite{GST_Prop2R}.

The current note has two modest goals:
\begin{enumerate}
	\item to advertise how work of L. Rudolph can be applied to braid conjugacy class invariants to in principle obstruct a slice knot from being ribbon,
	\item to explain why this will not work when the ideas in (1) are applied to the conjugacy class invariant defined and studied in \cite{dt},
\end{enumerate}

We are disappointed about (2) but hope that (1) may still prove useful.

\begin{figure}
\label{fig:ribbon}
\centering
\includegraphics[height=2in]{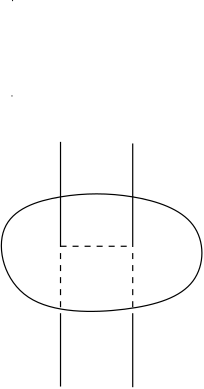}
\caption{A {\em ribbon type singularity } of an immersed surface $f: S \looparrowright S^3$ with $\partial S \neq \emptyset$ is an imbedded arc, $a \subset S^3$, of self-intersection whose preimage, $f^{-1}(a) \subset S$, has two imbedded connected components, one with boundary in $\partial S$ and one with boundary in $S \setminus \partial S$.}
\end{figure}

\subsection*{Acknowledgements}
This note is based on a talk the author delivered at the Georgia Topology Conference, 2017. She thanks Mark Hughes, whose questions following the talk prompted a deeper investigation, Kyle Hayden, with whom she studied the constructions of Rudolph described here, and Joan Birman, Tony Licata, and Stephan Wehrli for interesting conversations. 

\section{Obtaining ribbon obstructions from braid conjugacy class invariants}
Let $\Braid_n$ denote the $n$--strand braid group, with standard Artin generators $\sigma_1, \ldots \sigma_{n-1}$. In this paper, braid words will be read from left to right, and braid diagrams will be read (and oriented) from top to bottom, with the first strand on the left. Let $i_{n,n+k}: \Braid_n \rightarrow \Braid_{n+k}$ denote the standard inclusion of the $n$--strand braid group into the $(n+k)$--strand braid group that adjoins $k$ trivial strands on the right. We adopt Rudolph's conventions in \cite{Rudolph_SpecPos} and denote $i_{n,n+k}(\beta) \in \Braid_{n+k}$ by $\beta^{(k)}$. Let $\Braid := \bigcup_n \Braid_n$.

For $\beta \in \Braid_n$, we use $[\beta]$ to denote its conjugacy class in $\Braid_n$ and will denote its positive (resp., negative) Markov stabilization, $i_{n, n+1}(\beta)\sigma_n$ (resp., $i_{n,n+1}(\beta)\sigma_{n}^{-1}$), by $\beta^+$ (resp., by $\beta^-$).

For $\beta \in \Braid_n$, we use $\hbeta$ to denote its {\em braid-oriented} closure as a link in $S^3$ and $\hbetaB$ to denote its braid-oriented closure as a link in the complement of the braid axis, $B$. Of course, if $[\beta] = [\beta']$, then $\hbeta = \hbeta'$ as links in  $S^3$, but the converse is not generally true. On the other hand, $[\beta] = [\beta']$ iff $\hbetaB = \hbetaB'$.

We begin with a few definitions (cf. \cite{Rudolph_SpecPos}) in preparation for reminding the reader of two theorems of L. Rudolph.

\begin{definition} \label{defn:band} A positive (resp., negative) {\em band} in $\Braid_n$ is any braid $\beta \in \Braid_n$ for which $[\beta]= [\sigma_1]$ (resp., $[\beta] = [\sigma_1^{-1}]$). 
\end{definition}

\begin{remark} Any band can be written in the form $\omega \sigma_1 \omega^{-1}$, where $\omega$ is a word in the Artin generators. Note that $\sigma_i^\pm$ is a band for each $i$, via the conjugating word \[\omega = (\sigma_{i-1}\sigma_i)(\sigma_{i-2}\sigma_{i-1})\cdots(\sigma_1\sigma_2).\] 
\end{remark}

\begin{definition} \label{defn:bandrank} A {\em band presentation} of a braid $\beta$ is an explicit decomposition of $\beta$ as a product of bands: \[\beta = \prod_{j=1}^c \omega_j\sigma_{1}^{\pm}\omega_j^{-1},\] where, in addition, an explicit word in the Artin generators has been chosen for each of the conjugating braids, $\omega_j$. The {\em band rank} of a braid $\beta$, denoted $\rk(\beta)$, is the minimal $c \in \Z^{\geq 0}$ such that $\beta$ can be expressed as a product of $c$ bands. Note that $\rk(\beta)$ is a braid conjugacy class invariant.
\end{definition}

The nomenclature is justified by Rudolph's observation that if $\beta \in \Braid_n$ can be written as a product of $c$ bands, then the closure of $\beta$ bounds an obvious ribbon-immersed surface in $S^3$ with a disk ($0$--handle) for each strand of the braid and a band ($1$--handle) for each term in the product. This surface is in general not imbedded, since each band may intersect the $n$ disks, as well as itself. However, the resulting singularities will always be of ribbon type. By appropriately pushing neighborhoods of the ribbon singularities into $B^4$ one therefore obtains a properly-imbedded surface, Morse with respect to the radial function, with no maxima.

\begin{definition} \label{defn:braidedbanded} A {\em braided, banded surface} is a ribbon-immersed oriented surface in $S^3$ bounded by the closure of a band-presented braid $\beta \in \Braid_n$ as described above. Precisely, let $\beta = \prod_{j=1}^c \omega_j \sigma_1^{\pm}\omega_j^{-1}$ be a band presentation of $\beta$. Begin with the $n$ braid-oriented disjoint disks $D_1, \ldots, D_n$ bounded by the closure of the $n$--strand trivial braid $\Id_n$ and attach $c$ half-twisted bands $b_1, \ldots, b_c$ to $D_1, \ldots, D_n$ as follows.\footnote{The $j$th band, $b_j$, is a ribbon-immersed $2$--dimensional $1$--handle. We shall refer to the two attaching regions of each band as its {\em ends}. This should not be confused with the ends of the braid word $\omega_j\sigma_1^\pm\omega_j^{-1}$ defining the $j$th band. The geometric band is constructed by reading the braid word out from the center, rather than from left to right.} 
\begin{enumerate}
	\item One of the two ends of $b_j$ is attached to $D_1$. A neighborhood of this end looks like the standard oriented band bounded by $\sigma_1^\pm$ (with orientation induced by the orientation on $\sigma_1^{\pm}$).
	\item Construct the remainder of $b_j$ by extending the standard band above, reading the word $\omega_j$ backwards one letter at a time and forming the associated nested subconjugates of $\omega_j\sigma_1^{\pm}\omega_j^{-1}$. Note that as the band is extended in this way, it may intersect itself and the interiors of the disks $D_i$, but all of these intersections will be of ribbon type.
	\item Let $\pi:\Braid_n \rightarrow S_n$ be the standard projection of the braid group to the symmetric group on the same number of letters, and let $\pi_j$ be the image of $2$ under the permutation $\pi(\omega_j^{-1})$. Then the other end of $b_j$ will be attached to $D_{\pi_j}$.
\end{enumerate}
See Figure \ref{fig:Bands}.
\end{definition}

\begin{figure}
\label{fig:Bands}
\centering
\includegraphics[height=2.5in]{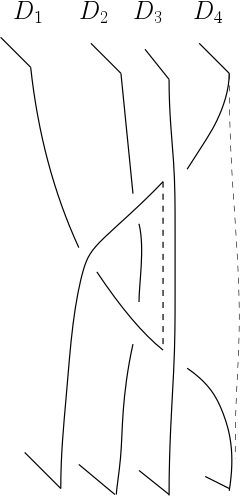}
\caption{The positively half-twisted band associated to the conjugate $\omega\sigma_1\omega^{-1} \in \Braid_4$, where $\omega = \sigma_3^{-1}\sigma_2$. One end of the band is attached to $D_1$, the other to $D_4$, and the band intersects $D_3$ in a ribbon singularity.}
\end{figure}

Theorems \ref{thm:ribbon} and \ref{thm:slice} and their corollaries are due to L. Rudolph.

\begin{theorem}\cite[Sec. 3]{Rudolph_BraidSurf} \label{thm:ribbon} Let $\Sigma \subset S^3$ be a ribbon-immersed orientable surface. Then $\Sigma$ is isotopic to a {\em braided, banded surface}. 
\end{theorem}

\begin{corollary} \label{cor:ribbon} If $K$ is ribbon, then for some $n$ there exists a braid $\beta \in \Braid_n$ with $\hbeta = K$ for which $\rk(\beta) = n-1$. 
\end{corollary}

\begin{proof} If $K$ is ribbon, then it bounds a ribbon-immersed orientable disk in $S^3$. By Theorem \ref{thm:ribbon}, this ribbon-immersed disk is isotopic to a braided, banded surface $S$ for some $\beta \in \Braid_n$ with $\hbeta = K$. This implies that $\beta$ has a band presentation with $n-1$ bands, so $\rk(\beta) \leq n-1$. But $\beta$ cannot be presented with fewer than $n-1$ bands, since this would force $\hbeta$ to be disconnected.
\end{proof}

\begin{theorem}\cite[Prop. 6.11]{Rudolph_SpecPos} \label{thm:slice} Let $\hbeta$ be a closed braid bounding a properly-imbedded oriented surface $S \subset B^4$. Then there exists $q \in \Z^{\geq 0}$ and some band presentation of $\beta^{(q)}$ for which the push-in to $B^4$ of the associated braided, banded surface is isotopic to $S$ with $q$ open disks removed.
\end{theorem}

\begin{corollary} \label{cor:slice} If $K$ is slice, then for some $n,q$ there exists a braid $\beta \in \Braid_n$ with $\hbeta = K$ for which \[\mbox{rk}_{n+q}\left(\beta^{(q)}\right) \leq (n-1)+2q.\]
\end{corollary}

\begin{proof} Let $K$ be slice, and choose $\beta \in \Braid_n$ with $\hbeta = K$. Then $\hbeta$ bounds a slice disk in $B^4$. Now Theorem \ref{thm:slice} tells us there exists a $q$ for which $\hbeta^{(q)}$ bounds a braided, banded surface of Euler characteristic $1-q$ (its push-in is the slice disk with $q$ open disks removed). The associated band presentation of $\hbeta^{(q)}$ therefore has $b$ bands, where \[(n+q) - b = 1-q.\] The result follows.
\end{proof}

\begin{proposition} Let $\phi: \Braid \rightarrow \R$ be a real-valued braid invariant with the following properties.
\begin{enumerate}
	\item $\phi(\beta) = \phi(\beta')$ if $[\beta] = [\beta']$,
	\item $\phi(\beta^{\pm}) = \phi(\beta) + 1$,
	\item $\phi(\beta) \leq \rk([\beta])$,
	\item For some $n,k$, $\exists \,\, \beta \in \Braid_n$ satisfying $\phi\left(\beta^{(k)}\right) \neq \phi(\beta) + 2k$,
	\item There exists a knot $K$ with a braid representative $\beta \in \Braid_n$ satisfying $\phi(\beta) >n-1$.
\end{enumerate}
Then $\phi(\beta) - n$ is an oriented link invariant that has the potential to obstruct a slice knot from being ribbon. 
\end{proposition}

\begin{proof} The Alexander and Markov theorems (cf. \cite{Birman_book}) tell us that any (oriented) link is isotopic to a (braid-oriented) closed braid and that any two closed braid representatives of the same link are related by a finite sequence of closed braid isotopies (conjugations) and Markov de/stabilizations. Properties (1) and (2) then imply that $\phi(\beta) - n$ is an oriented link invariant.

Property (5) tells us that there exists a knot $K$ with a braid representative $\beta \in \Braid_{n_0}$ (for some $n_0$) satisfying $n_0-1 < \phi(\beta)$.  Properties (1) and (2) and the Markov theorem then tell us that for {\em every} braid representative $\beta \in \Braid_n$ of $K$ (for {\em every} $n$), $\phi(\beta) > n-1$. Property (3) and Corollary \ref{cor:ribbon} then tells us that $K = \hbeta$ cannot be ribbon.

Finally, we must show that it's possible for $\phi$ to obstruct a {\em slice} knot from being ribbon. Let $K$ be a slice knot and $\beta \in \Braid_n$ a braid for which $\hbeta = K$. Then Corollary \ref{cor:slice} tells us that for some $q$ we have $\mbox{rk}_{n+q}\left(\hbeta^{(q)}\right) \leq (n-1)+2q.$ Property (3) then tells us that $\phi(\beta^{(q)}) \leq (n-1) + 2q$, and Property (4) tells us that it's possible for $\phi(\beta) > n-1$.
\end{proof}

\section{The annular Rasmussen invariant does not give new ribbon obstructions}

In \cite{dt} we define, for each braid $\beta \in \Braid$ a family, $d_t(\hbetaB)$ for $t \in [0,1]$, of real-valued braid conjugacy class invariants.\footnote{In \cite{dt}, $d_t(\hbetaB)$ is defined for $t \in [0,2]$. However, the relevant interval is $t \in [0,1]$ because the invariant is symmetric about $t=1$ (cf. \cite[Thm. 1]{dt}).} These braid conjugacy class invariants are defined using the structure of the annular Khovanov-Lee complex of $\hbeta$ as a $(\Z \oplus \Z)$--filtered chain complex.\footnote{In \cite{dt}, $d_t$ is defined as an invariant of oriented annular links, but it is noted there that the closure of a braid comes equipped with a standard orientation that we call the braid-like orientation, see \cite[Sec. 3.2]{dt}, and two braid-oriented braid closures are isotopic as annular links iff they are conjugate in $\Braid$.} We recall here some of its relevant properties.

\begin{theorem} \cite[Thm. 1, Cor. 3, Thm. 6]{dt} \label{thm:dtprops} Let $\beta \in \Braid_n$ be an n-strand braid equipped with the braid-like orientation.  Viewed as a function from $[0,1]$ to $\R$, $d_t(\hbetaB)$ satisfies the following properties:
\begin{enumerate}
	\item $d_t(\hbetaB)$ is a conjugacy class invariant,
	\item $d_t(\hbetaB)$ is continuous and piecewise linear, with slopes \[m_t(\hbetaB) \in \{-n, -n+2, \ldots, n-2, n\},\] 
%	\item $d_{1-t}(\beta) = d_{1+t}(\beta)$ for all $t \in [0,1]$,
	\item $d_0(\hbetaB) = s(\hbeta) - 1$,
	\item $d_1(\hbetaB) = w(\hbetaB)$
	\item $|d_t(\hbetaB) - d_t(\widehat{\Id}_n)| \leq \rk(\hbetaB)$
	\item $d_t(\widehat{\Id}_n) = n(t-1)$
\end{enumerate}
\end{theorem}

In the above, $s(\hbeta)$ is Beliakova-Wehrli's generalization \cite{BW} of Rasmussen's invariant of knots \cite{Rasmussen_Slice} to an invariant of oriented links in $S^3$, and $w(\hbetaB)$ is the writhe of $\hbetaB$ equipped with the braid-orientation. In the literature, $w(\hbetaB)$ is sometimes called the {\em exponent sum}.

Recall that we're trying to apply the lower bound on band rank coming from the invariant $d_t(\hbetaB)$ to find an obstruction to a slice knot $K$ being ribbon. Rudolph's Corollary \ref{cor:ribbon} tells us that if we can find a lower bound on the band rank that is well-behaved enough under Markov de/stabilization, it might yield an effective ribbon obstruction.

Unfortunately, Theorem \ref{thm:bestbound} below tells us that the best lower bound on band rank coming from the $d_t$ invariant occurs at either $t=0$ or $t=1$. Moreover, Lemmas \ref{lem:sbadbound} and \ref{lem:wbadbound} tell us that neither of these invariants is strong enough to obstruct a slice knot from being ribbon. 

Indeed, none of these invariants is strong enough to obstruct a knot of {\em finite concordance order} from being ribbon. Recall that the concordance order of an oriented knot $K$ is the minimal $n \in \Z^{\geq 0}$ for which the connected sum of $n$ copies of $K$ with itself is slice. Theorem \ref{thm:bestbound} combined with Lemmas \ref{lem:sbadbound} and \ref{lem:wbadbound} tell us that if $d_t(\hbetaB)$ obstructs $\hbeta = K$ from being ribbon, then $K$ has infinite order in the smooth knot concordance group. 

\begin{theorem} \label{thm:bestbound} Let $\beta \in \Braid_n$ be an $n$-strand braid whose closure, $K = \hbeta$, has finite concordance order. Then \[\max_{t\in [0,1]} |d_t(\hbetaB) - d_t(\widehat{\Id}_n)|\] occurs at either $t=0$ or $t=1$.
\end{theorem}

\begin{proof} Theorem \ref{thm:dtprops}, Property (6), tells us that $d_t(\widehat{\Id}_n) = nt - n$. 

Since $\hbeta$ has finite concordance order, $s(\hbeta) = 0$, hence $d_0(\hbetaB) = s(\hbeta) -1 = -1$, and hence when $t=0$, $|d_t(\hbetaB) - d_t(\widehat{\Id}_n)| = n-1$.  Since $d_t(\hbetaB)$ is piecewise linear and $-n \leq m_t(\hbetaB) \leq n$, we know that \[-1 - nt \leq d_t(\hbetaB) \leq -1 + nt.\]

Assume that $\max_{t \in [0,1]}|d_t(\hbetaB) - d_t(\widehat{\Id}_n)|$ does not occur at $t=0$. This tells us that $\exists \,\, t' \in (0,1)$ with $n-1 < |d_{t'}(\hbetaB) - (nt' -n)|$. Moreover, it must be the case that \[d_{t'}(\hbetaB) - (nt'-n) < 0,\] since $d_t(\hbetaB) \leq -1 + nt$, hence $d_t(\hbetaB) - (nt-n) \leq n-1.$

But since $d_0(\hbetaB) - d_0(\widehat{\Id}_n )>0$ and $d_{t'}(\hbetaB) - d_{t'}(\widehat{\Id}_n) < 0$, there exists at least one $t \in (0,1)$ with \[d_t(\hbetaB) = d_t(\widehat{\Id}_n) = nt-n.\] Let $t_0$ be the maximum such $t \in (0,1)$.

We now claim that on every closed interval $[a,b] \subset [t_0, 1]$ on which $d_t(\hbetaB)$ is differentiable on its interior, \[\max_{t \in [a,b]} (nt-n) - d_t(\hbetaB)\] occurs at $t=b$.

This is because $m_t(\hbetaB) \leq n$, so $(nt-n) - d_t(\hbetaB)$ is monotonically non-decreasing on any such differentiable interval.

It follows in this case that \[\max_{t \in [0,1]} |d_t(\hbetaB)-d_t(\widehat{\Id}_n)|\] occurs at $t=1$.

\end{proof}

\begin{lemma} \label{lem:sbadbound} Let $\beta \in \Braid_n$ be an $n$-strand braid whose closure, $\hbeta$, is a knot $K$. If $K$ has finite concordance order, then $|d_0(\hbetaB) - d_0(\widehat{\Id}_n)| = n-1$.
\end{lemma}

\begin{proof} If $K$ has finite concordance order, then $s(K) = d_0(\hbetaB) + 1 = 0$, by \cite{Rasmussen_Slice}. The result then follows from Theorem \ref{thm:dtprops}, Property (6).
\end{proof}

\begin{lemma} \label{lem:wbadbound} Let $\beta \in \Braid_n$ be an $n$-strand braid whose closure, $\hbeta$, is a knot $K$. If $K$ has finite concordance order, then $|d_1(\hbetaB) - d_1(\widehat{\Id}_n)| \leq n-1$.
\end{lemma}

\begin{proof} This follows immediately from Theorem \ref{thm:dtprops}, Properties (4) and (6), and \cite[Lem. 2.1]{Lisca}, which says that the absolute value of the {\em writhe} of an $n$--braid whose closure has finite concordance order cannot exceed $n-1$. Lisca's lemma in turn relies on the slice-Bennequin inequality \cite{Rudolph_QPSlice}.
\end{proof}

\bibliography{BraidedBanded}

\def\polhk#1{\setbox0=\hbox{#1}{\ooalign{\hidewidth
  \lower1.5ex\hbox{`}\hidewidth\crcr\unhbox0}}}
\begin{thebibliography}{10}

\bibitem{BW}
A.~Beliakova and S.~M. Wehrli.
\newblock Categorification of the colored {J}ones polynomial and {R}asmussen
  invariant of links.
\newblock {\em Canad. J. Math.}, 60(6):1240--1266, 2008.

\bibitem{Birman_book}
J.~S. Birman.
\newblock {\em Braids, links, and mapping class groups}.
\newblock Princeton University Press, Princeton, N.J.; University of Tokyo
  Press, Tokyo, 1974.
\newblock Annals of Mathematics Studies, No. 82.

\bibitem{Fox_Probs}
R.~H. Fox.
\newblock Some problems in knot theory.
\newblock In {\em Topology of 3-manifolds and related topics ({P}roc. {T}he
  {U}niv. of {G}eorgia {I}nstitute, 1961)}, pages 168--176. Prentice-Hall,
  Englewood Cliffs, N.J., 1962.

\bibitem{GST_Prop2R}
R.~E. Gompf, M.~Scharlemann, and A.~Thompson.
\newblock Fibered knots and potential counterexamples to the property 2{R} and
  slice-ribbon conjectures.
\newblock {\em Geom. Topol.}, 14(4):2305--2347, 2010.

\bibitem{dt}
J.~E. Grigsby, A.M. Licata, and S.~M. Wehrli.
\newblock Annular {K}hovanov-{L}ee homology, braids, and cobordisms.
\newblock math.GT/1612.05953, 2016.

\bibitem{Lisca}
P.~Lisca.
\newblock On 3-braid knots of finite concordance order.
\newblock {\em Trans. Amer. Math. Soc.}, 369(7):5087--5112, 2017.

\bibitem{Rasmussen_Slice}
J.~Rasmussen.
\newblock Khovanov homology and the slice genus.
\newblock {\em Invent. Math.}, 182(2):419--447, 2010.

\bibitem{Rudolph_BraidSurf}
L.~Rudolph.
\newblock Braided surfaces and {S}eifert ribbons for closed braids.
\newblock {\em Comment. Math. Helv.}, 58(1):1--37, 1983.

\bibitem{Rudolph_SpecPos}
L.~Rudolph.
\newblock Special positions for surfaces bounded by closed braids.
\newblock {\em Rev. Mat. Iberoamericana}, 1(3):93--133, 1985.

\bibitem{Rudolph_QPSlice}
Lee Rudolph.
\newblock Quasipositivity as an obstruction to sliceness.
\newblock {\em Bull. Amer. Math. Soc. (N.S.)}, 29(1):51--59, 1993.

\end{thebibliography}
\end{document}